\UseRawInputEncoding
\documentclass[12pt]{article}
\usepackage[centertags]{amsmath}
\usepackage{amsfonts}
\usepackage{amssymb}
\usepackage{latexsym}
\usepackage{amsthm}
\usepackage{newlfont}
\usepackage{graphicx}
\usepackage{listings}
\usepackage{booktabs}
\usepackage{abstract}
\date{}

\newlength{\defbaselineskip}
\newcommand{\setlinespacing}[1]%
           {\setlength{\baselineskip}{#1 \defbaselineskip}}

\newcommand{\N}{{\mathbb{N}}}

\newcommand{\actaqed}{\hfill $\actabox$}
{\medskip\noindent \textit{Proof of #1. }}%
{\actaqed \medskip}

\def\D{{\mathcal D}}

\def\C{{\mathcal C}}

\def\H{{\mathcal H}}
\def \Tr{\mathcal T}

\def \cN{\mathcal N}
\def \cR{\mathcal R}

\def \V{\mathcal V}
\def \cV{\mathcal V}

\def\R{{\mathbb R}}
\def\Z{\mathbb Z}

\def \Td{{\mathbb T}^d}
\def \T{\mathbb T}

\def \<{\langle}
\def\>{\rangle}
\def \L{\Lambda}
\def \La{\Lambda}
\def \ep{\epsilon}

\def \ff{\varphi}
\def\al{\alpha}
\def\bt{\beta}
\def\ga{\gamma}
\def\la{\lambda}

\def\bx{\mathbf x}

\def\bk{\mathbf k}
\def\bu{\mathbf u}
\def\bv{\mathbf v}
\def\bw{\mathbf w}

\def\bt{\beta}

\def \csp{\overline{\operatorname{span}}}

\newtheorem{Theorem}{Theorem}[section]
\newtheorem{Lemma}{Lemma}[section]
\newtheorem{Definition}{Definition}[section]
\newtheorem{Proposition}{Proposition}[section]
\newtheorem{Remark}{Remark}[section]

\newtheorem{Corollary}{Corollary}[section]
\numberwithin{equation}{section}

\newcommand{\be}{\begin{equation}}
\newcommand{\ee}{\end{equation}}

\begin{document}

\title{Rate of convergence of Thresholding Greedy Algorithms}
\author{  V.N. Temlyakov\thanks{ Steklov Mathematical Institute of Russian Academy of Sciences, Moscow, Russia; Lomonosov Moscow State University;  Moscow Center of Fundamental and Applied Mathematics;  University of South Carolina.}}
\maketitle
\begin{abstract}
{The rate of convergence of the classical Thresholding Greedy Algorithm with respect to bases is studied in this paper. We bound the error of approximation by the product of both norms --  the norm of $f$ and the $A_1$-norm of $f$. We obtain some results for greedy bases, unconditional bases, and quasi-greedy bases. In particular, we prove that our bounds for the trigonometric basis and for the Haar basis are optimal.
   }
\end{abstract}

\section{Introduction}
\label{I}

This paper is a follow up to the very recent paper \cite{VT196} where we found the rate of convergence of some standard greedy algorithms with respect to arbitrary dictionaries. The new ingredient of the paper \cite{VT196} is that we bounded the error of approximation by the product of both norms --  the norm of $f$ and the $A_1$-norm of $f$. Typically, only the $A_1$-norm of $f$ is used.  
In this paper we focus on a special class of dictionaries, namely, on different types of bases -- greedy bases, unconditional bases, and quasi-greedy bases. We study 
the rate of convergence of the classical Thresholding Greedy Algorithm with respect to bases in 
the style of paper \cite{VT196}. 

 We consider here approximation in uniformly smooth real Banach spaces. For a Banach space $X$ we define the modulus of smoothness
$$
\rho(u) := \rho(u,X):= \sup_{\|x\|=\|y\|=1}(\frac{1}{2}(\|x+uy\|+\|x-uy\|)-1).
$$
A uniformly smooth Banach space is  one with the property
$$
\lim_{u\to 0}\rho(u)/u =0.
$$
We consider here the case $X=L_p$. Let $\Omega$ be a compact subset of $\R^d$ with the probability measure $\mu$. By $L_p$ norm, $1\le p< \infty$, of a function defined on $\Omega$,  we understand
$$
\|f\|_p:=\|f\|_{L_p(\Omega,\mu)} := \left(\int_\Omega |f|^pd\mu\right)^{1/p}.
$$
By $L_\infty$ norm we understand the uniform norm of continuous functions
$$
\|f\|_\infty := \max_{\bx\in\Omega} |f(\bx)|
$$
and with a little abuse of notations we sometimes write $L_\infty(\Omega)$ for the space $\C(\Omega)$ of continuous functions on $\Omega$.

It is well known (see for instance \cite{DGDS}, Lemma B.1) that in the case $X=L_p$, 
$1\le p < \infty$ we have
\be\label{Lprho}
\rho(u,L_p) \le \begin{cases} u^p/p & \text{if}\quad 1\le p\le 2 ,\\
(p-1)u^2/2 & \text{if}\quad 2\le p<\infty. \end{cases} 
\ee 

 We say that a set of elements (functions) $\D$ from $X$ is a dictionary if each $g\in \D$ has norm bounded by   one ($\|g\|\le 1$), and $\csp \D =X$.  
We define the best $m$-term approximation with regard to $\D$ as follows:
$$
\sigma_m(f):=\sigma_m(f,\D)_X :=\inf_{g_1,\dots,g_m} \inf_{c_k} \|f -\sum_{k=1}^m c_kg_k\|_X,
$$
where the infimum is taken over coefficients $c_k$ and sets of $m$ elements $g_1,\dots,g_m$, $g_i\in\D$, $i=1,\dots,m$. 

Let $\D^\pm := \{\pm g\, :\, g\in \D\}$. Denote  the closure of the convex hull of $\D^\pm$   by $A_1(\D)$.   For $f\in X$ we associate the following norm
$$
\|f\|_{A_1(\D)} := \inf \{M:\, f/M\in A_1(\D)\}.
$$
Clearly, $\|f\|_X\le \|f\|_{A_1(\D)}$. 

The following result is the corollary of known results on greedy approximation (see \cite{VT196}, Remark 3.2). 

\begin{Proposition}\label{IP1} Let $X$ be a uniformly smooth Banach space with $\rho(u,X) \le \ga u^q$, $1<q\le 2$. Then for any dictionary $\D$ and any $\al \in [0,1]$ we have for all $f\in X$ such that $\|f\|_{A_1(\D)} <\infty$, $f\neq 0$
\be\label{I1}
\frac{\sigma_m(f,\D)_{X}}{\|f\|_{X}^{1-\al}\|f\|_{A_1(\D)}^\al } \le C(\ga,q) m^{-\al(1-1/q)} .
\ee
\end{Proposition}

Denote $p^* := \min(p,2)$. Then Proposition \ref{IP1} and relation (\ref{Lprho}) imply the following Corollary \ref{IC1}.

\begin{Corollary}\label{IC1} Let $1<p<\infty$.  Then for any dictionary $\D\subset L_p(\Omega,\mu)$ and any $\al \in [0,1]$ we have for all $f\in L_p(\Omega,\mu)$ such that $\|f\|_{A_1(\D)} <\infty$, $f\neq 0$
\be\label{I1c}
\frac{\sigma_m(f,\D)_{p}}{\|f\|_{p}^{1-\al}\|f\|_{A_1(\D)}^\al } \le C(p) m^{-\al(1-1/p^*)} .
\ee
\end{Corollary}

Let a Banach space $X$, with a basis (Schauder basis) $\Psi =\{\psi_k\}_{k=1}^\infty$,    be given. We assume that $0<c_0\le \|\psi_k\|_X\le 1$, $k=1,2,\dots$,  and consider the following theoretical greedy algorithm. For a given element $f\in X$ we consider the expansion
\begin{equation}\label{2c1.1}
f = \sum_{k=1}^\infty c_k(f,\Psi)\psi_k. 
\end{equation}
For an element $f\in X$ we say that a permutation $\varrho$ of the positive integers   is  decreasing  if 
\begin{equation}\label{2c1.2}
|c_{k_1}(f,\Psi) |\ge |c_{k_2}(f,\Psi) | \ge \dots ,  
\end{equation}
 where $\varrho(j)=k_j$, $j=1,2,\dots$, and write $\varrho \in D(f)$.
If the   inequalities are strict in (\ref{2c1.2}), then $D(f)$ consists of only one permutation. We define the $m$th greedy approximant of $f$, with regard to the basis $\Psi$ corresponding to a permutation $\varrho \in D(f)$, by the formula
$$
G_m(f):=G_m(f,\Psi)  :=G_m(f,\Psi,\varrho) := \sum_{j=1}^m c_{k_j}(f,\Psi)\psi_{k_j}.
$$

The above algorithm $G_m(\cdot,\Psi)$ is a simple algorithm, which describes the theoretical scheme  for $m$-term approximation of an element $f$. This algorithm is called the  Thresholding Greedy Algorithm (TGA) or simply the Greedy Algorithm (GA). In order to understand the efficiency of this algorithm we compare its accuracy with the best-possible accuracy when an approximant is a linear combination of $m$ terms from $\Psi$.   The best we can achieve with the algorithm $G_m$ is
$$
\|f-G_m(f,\Psi,\varrho)\|_X = \sigma_m(f,\Psi)_X,
$$
or the slightly weaker
\begin{equation}\label{2c1.3}
\|f-G_m(f,\Psi,\varrho)\|_X \le G\sigma_m(f,\Psi)_X,  
\end{equation}
for all elements $f\in X$, and with a constant $G=C(X,\Psi)$ independent of $f$ and $m$.
It is clear that, when $X=H$ is a Hilbert space and $\mathcal B$ is an orthonormal basis, we have
$$
\|f-G_m(f,\mathcal B,\varrho)\|_H = \sigma_m(f,\mathcal B)_H.
$$

We consider the following asymptotic characteristic of the TGA with respect to a given basis $\Psi$
$$
\ga_m(\al,X,\Psi) := \sup_{f\in X: \|f\|_{A_1(\Psi)}<\infty} \sup_{\varrho\in D(f)} \frac{\|f-G_m(f,\Psi,\varrho)\|_X}{\|f\|_X^{1-\al} \|f\|_{A_1(\Psi)}^\al},\quad \al\in [0,1].
$$
We mostly concentrate on the case $X=L_p$, $1\le p\le \infty$. In such a case we write $\ga_m(\al,p,\Psi)$.  

 The following concept of a greedy basis has been introduced in \cite{KonT}.

\begin{Definition}\label{ID1}
We call a basis $\Psi$ a greedy basis if for
every $f\in X$ there exists a permutation $\varrho \in D(f)$ such
that
\be\label{gr}
\|f-G_m(f,\Psi,\varrho)\|_X\leq C\sigma_m (f,\Psi)_X
\ee
with a constant $C$ independent of $f$ and $m$.
\end{Definition}
It is known (see \cite{KonT}) that for a greedy basis $\Psi$ inequality (\ref{gr}) holds for all permutations $\varrho \in D(f)$. Therefore, 
Definition \ref{ID1} and Corollary \ref{IC1}  imply the following bounds for the greedy bases. 

\begin{Theorem}\label{IT1} Let $1<p<\infty$.  Then for any greedy basis $\Psi$ of $L_p(\Omega,\mu)$ and any $\al \in [0,1]$ we have  
\be\label{I7}
\ga_m(\al,L_p(\Omega,\mu),\Psi) \le C(p,\Psi) m^{-\al(1-1/p^*)} .
\ee
\end{Theorem}

In this paper we extend Theorem \ref{IT1} in different directions. In Section \ref{uq} we consider the Haar basis, which is a greedy basis for the $L_p$, $1<p<\infty$, (see \cite{T14}). We find the right order of 
the $\ga_m(\al,L_p([0,1],\mu),\Psi)$ in the case of $\Psi$ being the Haar basis $\H_p$ and $\mu=\la$ being the Lebesgue measure. Those bounds are better than the corresponding bounds in Theorem \ref{IT1} in the case $2<p<\infty$. The reader can find results on greedy approximation of smooth multivariate functions with respect to bases similar to the multivariate Haar basis  in \cite{VT69} and \cite{VT77}.

We consider bases, which are not greedy bases. In Section \ref{Tr} we find the right order of the quantity $\ga_m(\al,L_p(\T^d,\la),\Psi)$ 
in the case, when $\Psi$ is the trigonometric system. The reader can find a survey of results on the rate of best $m$-term trigonometric approximations of smooth multivariate functions in \cite{VTbookMA}, Ch.9.

It is known (see \cite{KonT}) that the greedy bases are 
exactly those bases, which are unconditional and democratic (see the definitions below in Section \ref{uq}). In Section \ref{uq} (see Theorems \ref{UT1} and \ref{UT1a}) we prove that Theorem \ref{IT1}
holds under assumption that $\Psi$ is an unconditional basis. 

In Subsection \ref{Q} we discuss a wider 
class of bases than the class of unconditional bases -- the quasi-greedy bases. It turns out that Theorem
\ref{IT1} can be extended to the class of quasi-greedy bases (see Theorem \ref{QT2} below). Finally, in 
Subsection \ref{Q} we prove that some error bounds can be improved for a subclass of quasi-greedy bases, namely, for the uniformly bounded quasi-greedy bases. 

We use $C$, $C(p,d)$, {\it etc.}, to denote various constants, the arguments indicating dependence on other parameters. We use the following symbols for brevity. For two non-negative sequences $a=\{a_n\}_{n=1}^\infty$ and $b=\{b_n\}_{n=1}^\infty$, the relation, or order inequality, $a_n \ll b_n$ means that there is a number $C(a,b)$ such that, for all $n$, we have $ a_n\le C(a,b)b_n$; and the relation $a_n\asymp b_n$ means that $a_n\ll b_n$ and $b_n\ll a_n$. 

\section{Trigonometric system}
\label{Tr}

In this section  we consider the case $X=L_p({\mathbb T}^d,\la)$, $1\le p \le \infty$, $\la$ is the normalized Lebesgue measure on $\T^d$, $\Psi ={\mathcal T}^d := \{e^{i(\bk,\bx)}\}_{\bk\in {\mathbb Z}^d}$ is the trigonometric system. It is convenient for us to deal with the complex trigonometric system. Results of this section hold for the real trigonometric system as well. Moreover, note that we can use the general results for real Banach spaces and real trigonometric system and derive from those results the corresponding results for the complex trigonometric system. We define
$$
\|f\|_{A_1(\Tr^d)}:= \sum_{\bk\in \Z^d} |\hat f(\bk)|,\quad \hat f(\bk) := (2\pi)^{-d}\int_{\T^d} f(\bx)e^{-i(\bk,\bx)}d\bx.
$$

The following Theorem \ref{TrT1} 
is known (see \cite{T1} and \cite{VTsparse}, p.25, Theorem 2.2.1).

\begin {Theorem}\label{TrT1}  
For each $f \in L_p(\Td)$ we have
$$
\|f - G_m(f,{\mathcal T}^d)\|_p \le (1+3m^{h(p)})\sigma_m(f,{\mathcal T}^d)_p , \quad 1 \le p \le \infty,
$$
where $h(p) := |1/2-1/p|.$
\end{Theorem}

We begin with the case $2\le p< \infty$. 

\begin{Theorem}\label{TrT2} Let $2\le p < \infty$. Then
$$
\ga_m(\al,p,\Tr^d) \asymp m^{h(p)-\al/2},\quad \al\in [0,1].
$$
\end{Theorem}
\begin{proof} The upper bound follows from Theorem \ref{TrT1} and Corollary  \ref{IC1}.  We now prove the lower bound. Clearly, it is sufficient to carry out the proof in the 
case $d=1$. We use  
  the Rudin-Shapiro polynomials (see, for instance, \cite{VTsparse}, Appendix, Section 7.4):
\begin{equation}\label{Tr1}
{\mathcal R}_m(x) = \sum_{|k| \le m} \epsilon_k e^{ikx}, \quad \epsilon_k = \pm 1,
\quad x \in \T,  
\end{equation} 
which satisfy the estimate
\begin{equation}\label{Tr2}
\|{\mathcal R}_m\|_\infty \le C_1m^{1/2} ,  
\end{equation}
with an absolute constant $C_1$. Denote for  $s = \pm 1$
$$
\La_{s} := \{ k : \epsilon_k = s \}.
$$  
Let $s = \pm 1$ be such that $|\La_s| > |\La_{-s}|$. Then $|\La_{-s}|\le m$.  Let $G_m(\cR_m)$ be a realization of the TGA  applied to $f=\cR_m$ with all frequencies in $\La_{-s}$ and maybe with some 
frequencies in $\La_s$. Denote $f_m := \cR_m - G_m(\cR_m)$. Then all frequencies of $f_m$ are 
in $\La_s$ and, therefore,
\begin{equation}\label{Tr5}
\|f_m\|_\infty \ge  m+1 . 
\end{equation}
 By the Nikol'skii inequality for trigonometric polynomials (see \cite{VTsparse}, Appendix, p.253, Theorem 7.5.4) 
 the relation (\ref{Tr5}) implies
\begin{equation}\label{Tr7}
\|f_m\|_p \ge C_2m^{-1/p}\|f_m\|_\infty 
\ge C_2m^{1-1/p} .  
\end{equation}
Clearly,   we have
\be\label{Tr8}
\|\mathcal R_{m}\|_{A_1(\Tr)} = 2m+1.
\ee
Comparing (\ref{Tr2}), (\ref{Tr7}), and (\ref{Tr8}) we get the required estimate in the case 
$2 \le p < \infty$.
\end{proof}
The following Corollary \ref{TrC1} might be of independent interest.
\begin{Corollary}\label{TrC1} Let $2\le p< \infty$. Then for any $f\in L_p$ and for each $m\in \N$ we have
$$
\|G_m(f,\Tr^d)\|_p \le C(p)\|f\|_p^{2/p}\|f\|_{A_1(\Tr^d)}^{1-2/p}.
$$
\end{Corollary}
\begin{proof} It follows from Theorem \ref{TrT2} with $\al = 2h(p)$ and an observation that the constant 
in the upper bound of Theorem \ref{TrT2} may depend only on $p$. 
\end{proof} 

We now proceed to the case $1\le p\le 2$.

\begin{Theorem}\label{TrT3} Let $1\le p \le 2$. Then
$$
\ga_m(\al,p,\Tr^d) \asymp m^{h(p)-\al/p},\quad \al\in [0,1].
$$
\end{Theorem}
\begin{proof} We begin with the upper bounds. In the case $1<p\le 2$ we can follow along the lines of the proof of 
Theorem \ref{TrT2} and apply  Theorem \ref{TrT1} and Corollary \ref{IC1}. This gives 
$$
\ga_m(\al,p,\Tr^d) \le C(p) m^{h(p)-\al/p'},\quad p':= \frac{p}{p-1}.
$$
It is clear that for all $\al\in(0,1]$ and $p<2$ we have 
$$
h(p)-\al/p < h(p)-\al/p'.
$$
Therefore, we need a different argument. Denote $f_m:= f-G_m(f,\Tr^d)$. Then by Theorem \ref{TrT1} 
we obtain
\be\label{Tr9}
\|f_m\|_p \le (1+3m^{h(p)})\|f\|_p.
\ee
We now need the following well known simple lemma from \cite{Tmon}, p.92 (see \cite{DTU}, Section 7.4, for a historical discussion).  
\begin{Lemma}\label{mLqp} Let $a_1\ge a_2\ge \cdots \ge a_M\ge 0$ and $1\le q\le p\le \infty$.
Then for all $m\le M$ one has
$$
\left(\sum_{k=m}^M a_k^p\right)^{1/p} \le m^{-\bt} \left(\sum_{k=1}^M a_k^q\right)^{1/q},\quad \bt:= 1/q-1/p.
$$
\end{Lemma}
By Lemma \ref{mLqp} with $q=1$ and $p=2$ we obtain
\be\label{Tr10}
\|f_m\|_p \le \|f_m\|_2 \le m^{-1/2}\|f\|_{A_1(\Tr^d)}.
\ee
Writing $\|f_m\|_p = \|f_m\|_p^{1-\al}\|f_m\|_p^\al$ and using (\ref{Tr9}) and (\ref{Tr10}) we obtain the upper bound
with an absolute constant as an extra factor. 

We now prove the lower bounds. We use the de la Vall\'ee Poussin kernels (see, for instance, \cite{VTsparse}, Appendix, Section 7.4) for $m=1,2,\dots$
\begin{equation}\label{Tr11}
\V_m(x) := \frac {1}{m} \sum_{l=m}^{2m-1} \D_l(x) ,\quad \D_l(x):= \sum_{|k|\le l} e^{ikx}, \quad x \in \T. 
\end{equation}
It is known that (see, for instance, \cite{VTsparse}, Appendix, p.246, (7.4.14))
\begin{equation}\label{Tr12}
\|\V_m\|_p \le Cm^{1-1/p} ,\quad m=1,2,\dots , \quad 1 \le p \le \infty .  
\end{equation}
We use the notations from the above proof of Theorem \ref{TrT2}. Let $G_m(\cV_m)$ be a realization of the TGA  applied to $f=\cV_m$ with all frequencies in $\La_{-s}$ and maybe with some 
frequencies in $\La_s$. Denote $f_m := \cV_m - G_m(\cV_m)$. Then a lower bound for $\|f_m\|_p$ follows
from the relations
$$
m \le |\<f_m,{\mathcal R}_m\>| \le \|f_m\|_p \|{\mathcal R}_m\|_{p'} \le C_1m^{1/2}\|f_m\|_p .
$$
We obtain
\begin{equation}\label{Tr15}
\|f_m\|_p \ge C_3m^{1/2} . 
\end{equation}
Clearly,   we have
\be\label{Tr14}
\|\mathcal V_{m}\|_{A_1(\Tr)} \le 4m.
\ee
Note, that it is easy to check that actually $\|\mathcal V_{m}\|_{A_1(\Tr)} = 3m$. 
The relations   (\ref{Tr12}), (\ref{Tr14}), and (\ref{Tr15}) imply the required inequality in the case 
$1 \le p \le 2$. The proof of Theorem \ref{TrT3} is now complete. 
\end{proof}

\begin{Corollary}\label{TrC2} Let $1\le p\le 2$. Then for any $f\in L_p$ and for each $m\in \N$ we have
$$
\|G_m(f,\Tr^d)\|_p \le C\|f\|_p^{p/2}\|f\|_{A_1(\Tr^d)}^{1-p/2}.
$$
\end{Corollary}

\section{Unconditional and quasi-greedy bases}
\label{uq}

We recall that we assume that $\Psi :=\{\psi_n\}_{n=1}^\infty$ is a basis (Schauder basis) of a Banach space $X$ satisfying the condition $\|\psi_n\|_X \le 1$, $n=1,2,\dots$. 

\subsection{Unconditional bases} 
\label{U}

There are different equivalent definitions of the concept of {\it unconditional basis} (see, for instance, \cite{VTsparse}, Section 1.4). It is convenient for us to use the following two
equivalent definitions. We use the representation
$$
f = \sum_{n=1}^\infty c_n(f)\psi_n.
$$

{\bf (I).} $\Psi$ is an unconditional basis if for any subset of integers $\La\subset \N$ there is a bounded linear projection $P_\La$ defined by
$$
P_\La(f) := \sum_{n\in \L} c_n(f)\psi_n,  
$$
and 
$$
K:=K(\Psi):=\sup_\La\|P_\La\| <\infty .
$$

{\bf (II).} $\Psi$ is an unconditional basis if for every choice of signs $\theta:=\{\theta_n\}_{n=1}^\infty$ 
we have a bounded linear operator $M_\theta$ defined by
$$
M_\theta(f) := \sum_{n=1}^\infty \theta_nc_n(f)\psi_n,  
$$
and $\sup_\theta\|M_\theta\|<\infty$.

It is well known that for an unconditional basis $\Psi$ of the space $L_p(\Omega,\mu)$, $1<p<\infty$,
we have the following properties.

{\bf U1.} There are two positive constants $C_i(\Psi,p)$, $i=1,2$, such that for any $f\in L_p(\Omega,\mu)$ we have
$$
C_1(\Psi,p) \|f\|_p \le \|S(f,\Psi)\|_p \le C_2(\Psi,p) \|f\|_p,\quad S(f,\Psi) :=\left( \sum_{n=1}^\infty |c_n(f)\psi_n|^2\right)^{1/2}.
$$

{\bf U2.} Property U1 implies the following inequalities. Let $p^* := \min(p,2)$, then for $f\in L_p(\Omega,\mu)$ we have
$$
\|f\|_p \le C_3(\Psi,p) \left(\sum_{n=1}^\infty \|c_n(f)\psi_n\|_p^{p^*}\right)^{1/p^*},\quad 1<p<\infty.
$$

\begin{Theorem}\label{UT1} Let $2\le p < \infty$ and let $\Psi$ be an unconditional basis of $L_p(\Omega,\mu)$. Then
$$
\ga_m(\al,p,\Psi) \ll m^{-\al/2},\quad \al\in [0,1].
$$
\end{Theorem}
\begin{proof} By the definition (I) we get for $f_m := f-G_m(f,\Psi)$
\be\label{U1}
\|f_m\|_p \le K \|f\|_p.
\ee
By the property U2 we obtain
$$
\|f_m\|_p \le C_3(\Psi,p) \left(\sum_{n\notin \La_m(f)} \|c_n(f)\psi_n\|_p^{2}\right)^{1/2}
$$
\be\label{U2}
  \le C_3(\Psi,p) \left(\sum_{n\notin \La_m(f)} |c_n(f)|^{2}\right)^{1/2},
\ee
where $\La_m(f)$ is such that
$$
\min_{n\in \La_m(f)}|c_n(f)| \ge \max_{n\notin \La_m(f)}|c_n(f)|.
$$
Therefore, by Lemma \ref{mLqp} inequality (\ref{U2}) implies
\be\label{U3}
\|f_m\|_p \le C_3(\Psi,p) m^{-1/2}\|f\|_{A_1(\Psi)}.
\ee
Writing $\|f_m\|_p = \|f_m\|_p^{1-\al}\|f_m\|_p^\al$ and using (\ref{U1}) and (\ref{U3}) we obtain 
the required in Theorem \ref{UT1} inequality.
\end{proof}

\begin{Theorem}\label{UT1a} Let $1< p \le 2$ and let $\Psi$ be an unconditional basis of $L_p(\Omega,\mu)$. Then
$$
\ga_m(\al,p,\Psi) \ll m^{-\al(1-1/p)},\quad \al\in [0,1].
$$
\end{Theorem}
\begin{proof} This proof goes along the lines of the proof of Theorem \ref{UT1}. We only point out where we use another argument. Instead of (\ref{U2}) by the property U2 we obtain
$$
\|f_m\|_p \le C_3(\Psi,p) \left(\sum_{n\notin \La_m(f)} \|c_n(f)\psi_n\|_p^{p}\right)^{1/p}
$$
\be\label{U2a}
  \le C_3(\Psi,p) \left(\sum_{n\notin \La_m(f)} |c_n(f)|^{p}\right)^{1/p}.
\ee 
By Lemma \ref{mLqp} with $q=1$ inequality (\ref{U2a}) implies
\be\label{U3a}
\|f_m\|_p \le C_3(\Psi,p) m^{-(1-1/p)}\|f\|_{A_1(\Psi)}.
\ee
Writing $\|f_m\|_p = \|f_m\|_p^{1-\al}\|f_m\|_p^\al$ and using (\ref{U1}) and (\ref{U3a}) we obtain 
the required in Theorem \ref{UT1a} inequality.

\end{proof} 

Denote $\H :=\{H_k\}_{k=1}^\infty$ the Haar basis on $[0,1)$ normalized in $L_2([0,1],\la)$ with respect to the Lebesgue measure $\la$: $H_1 =1$ on $[0,1)$ and for $k=2^n +l$, $n=0,1,\dots$, $l=1,2,\dots,2^n$
$$
H_k = \begin{cases} 2^{n/2}, & x \in [(2l-2)2^{-n-1}, (2l-1)2^{-n-1})\\
-2^{n/2}, & x \in [(2l-1)2^{-n-1}, 2l2^{-n-1}) \\
0, & \mbox{otherwise}.\end{cases}
$$
We denote by $\H_p :=\{H_{k,p}\}_{k=1}^\infty$ the Haar basis $\H$ renormalized in $L_p([0,1],\la)$.
It is well known that the Haar basis $\H_p$ is an unconditional basis for the $L_p([0,1],\la)$, $1<p<\infty$
(see, for instance, \cite{KS}). We now show that Theorem \ref{UT1} can be improved for the Haar basis.

\begin{Theorem}\label{UT2} Let $1< p < \infty$.  Then for the space $L_p([0,1],\la)$ we have
$$
\ga_m(\al,p,\H_p) \asymp m^{-\al(1-1/p)},\quad \al\in [0,1].
$$
\end{Theorem}
\begin{proof} We need the following simple technical lemma.

\begin{Lemma}\label{NL} For $b\in (0,1)$ and $m\in \N$ we have
$$
\sum_{n=1}^\infty \frac{1}{(n+m)n^b} \le C(b)m^{-b}.
$$
\end{Lemma}
\begin{proof} For $m=1$ the inequality holds. For $m\ge 2$ we have
$$
\sum_{n=1}^\infty \frac{1}{(n+m)n^b} \le  \sum_{n=1}^{m-1} \frac{1}{(n+m)n^b}  +\sum_{v=1}^\infty  \sum_{n=vm}^{(v+1)m-1} \frac{1}{(n+m)n^b}
$$
$$
\le \frac{1}{m}  \sum_{n=1}^{m-1} \frac{1}{n^b} + \sum_{v=1}^\infty \frac{m}{(v+1)m(vm)^b} \le C(b) m^{-b}.
$$
\end{proof}

The Haar basis is an unconditional basis and therefore for $1<p<\infty$
\be\label{U4}
\|f_m\|_p \le C_1(p) \|f\|_p.
\ee
We now use Lemma \ref{QL1} (see below), which holds for a quasi-greedy basis that is a weaker condition than the unconditionality of a basis. By Lemma \ref{QL1}  we obtain
\be\label{U5}
\|f_m\|_p \le C_2(p) \sum_{n=1}^\infty a_{n+m}(f)\ff(n)\frac{1}{n},
\ee
where $\ff(n)$ is the fundamental function of the Haar basis $\H_p$. It was proved in 
\cite{T14} (see also \cite{VTsparse}, p.35, Lemma 2.3.3) that
\be\label{U6}
\ff(n,\H_p,L_p([0,1],\la)) \le C_3(p) n^{1/p},\quad 1<p<\infty.
\ee
Relations (\ref{U5}), (\ref{U6}),  $a_n(f) \le n^{-1}\|f\|_{A_1(\H_p)}$, and Lemma \ref{NL} imply
\be\label{U7}
\|f_m\|_p \le C_4(p) m^{1/p-1}\|f\|_{A_1(\H_p)}.
\ee
Bounds (\ref{U4}) and (\ref{U7}) imply the upper bound in Theorem \ref{UT2}.

We now prove the lower bounds. For a given $m\in \N$ set
$$
f= \sum_{k=1}^{2m} H_{k,p}.
$$
Then by (\ref{U6})
\be\label{U8}
\|f\|_p \le C_3(p) (2m)^{1/p},\quad 1<p<\infty.
\ee
For any realization $G_m(f,\H_p)$ of the TGA we obtain by Lemma 2.3.4 from \cite{VTsparse}, p.35,
\be\label{U9}
\|f-G_m(f,\H_p)\|_p \ge C_5(p) m^{1/p}.
\ee
Clearly,
\be\label{U10}
\|f\|_{A_1(\H_p)} =2m.
\ee
The combination of relations (\ref{U8})--(\ref{U10}) proves the lower bounds in Theorem \ref{UT2}. 

\end{proof}

We say that $\Psi =\{\psi_k\}_{k=1}^\infty$ is $L_p$-equivalent to $\Phi =\{\phi_k\}_{k=1}^\infty$ if for any finite set $\La$ and any coefficients $c_k$, $k\in \La$, we have
\begin{equation}\label{U11}
C_1(p,\Psi,\Phi) \|\sum_{k\in \L}c_k\phi_k\|_{p} \le  \|\sum_{k\in \L}c_k\psi_k\|_{p} 
\le C_2(p,\Psi,\Phi) \|\sum_{k\in \L}c_k\phi_k\|_{p}  
\end{equation}
with two positive constants $C_1(p,\Psi,\Phi), C_2(p,\Psi,\Phi)$ which may depend on $p$, $\Psi$, and $\Phi$. For sufficient conditions on $\Psi$ to be $L_p$-equivalent to $\H_p$ see \cite{FJ} and \cite{DKT}. 

\begin{Remark}\label{UR1} Theorem \ref{UT2} holds for any basis $\Psi$, which is $L_p$-equivalent to the Haar basis $\H_p$. 
\end{Remark}

\subsection{Quasi-greedy bases}
\label{Q}

Let as above $X$ be an
infinite-dimensional separable Banach space with a norm
$\|\cdot\|:=\|\cdot\|_X$ and let $\Psi:=\{\psi_k
\}_{k=1}^{\infty}$ be a semi-normalized basis for $X$ ($0<c_0\le\|\psi_k\|\le 1,k\in
{\mathbb N}$). In this subsection we concentrate on a wider class of bases than unconditional bases -- quasi-greedy bases. The concept of quasi-greedy basis was introduced in \cite{KonT}. The reader can find a detailed discussion of quasi-greedy bases in \cite{VTsparse}, Ch.3.

\begin{Definition}\label{QD1}
The basis $\Psi$ is called quasi-greedy basis of $X$ if there exists a
constant $C=C(\Psi,X)$ such that
$$\sup_m \|G_m(f,\Psi)\|_X \leq C\|f\|_X.$$
\end{Definition}

Subsequently, Wojtaszczyk \cite{W1} proved that these are
precisely the bases for which the TGA merely converges, i.e.,
$$\lim_{n\rightarrow \infty}G_n(f)=f.$$

We now formulate some known results, which are useful in our applications. For a given element $f\in X$ we consider the expansion
$$
f=\sum_{k=1}^{\infty}c_k(f)\psi_k.
$$
Let a sequence $k_j,j=1,2,...,$ of the
positive integers be such that
$$
|c_{k_1}(f)|\geq |c_{k_2}(f)|\geq... \,\,.
$$
We use the notation
\be\label{an}
a_j(f):=|c_{k_j}(f)|
\ee
for the decreasing rearrangement of the coefficients of $f$.

The following theorem is from \cite{TYY1} (see also \cite{VTsparse}, p.70, Theorem 3.2.10). We note that in the case $p=2$
 Theorem \ref{QT1} was proved in \cite{W1}
\begin{Theorem}\label{QT1} Let $\Psi=\{\psi_k\}_{k=1}^\infty$ be a quasi-greedy basis of the $L_p$ space, $1<p<\infty$. Then there exist positive constants $C_i=C_i(p,\Psi)$, $i=1,\dots,4$, such that for each $f\in L_p$ we have
$$
C_1\sup_n n^{1/p}a_n(f)\le  \|f\|_p\le C_2 \sum_{n=1}^\infty n^{-1/2}a_n(f),\quad 2\le p <\infty;
$$
$$
C_3\sup_n n^{1/2}a_n(f)\le  \|f\|_p\le C_4 \sum_{n=1}^\infty n^{1/p-1}a_n(f),\quad 1< p \le 2.
$$
\end{Theorem}

Define   the  fundamental function $\varphi(m):=\varphi(m,\Psi,X)$ of a basis $\Psi$ in $X$
  as
$$
 \varphi(m,\Psi,X):=\sup_{|A|\le m}\|\sum_{k\in A}\psi_k\|.
 $$ 
 
 \begin{Definition}\label{QD2} We say that a basis $\Psi =\{\psi_k\}_{k=1}^\infty$ is a democratic basis for $X$ if there exists a constant $D:=D(X,\Psi)$ such that for any two finite sets of indices $P$ and $Q$ with the same cardinality $|P|=|Q|$ we have
$$
\|\sum_{k\in P} \psi_k\| \le D\|\sum_{k\in Q} \psi_k\|. 
$$
\end{Definition}
 
 The following Lemma \ref{QL1} was proved in \cite{VTsparse}, p.72, in the same way as Theorem \ref{QT1}.
\begin{Lemma}\label{QL1} Let $\Psi$ be a quasi-greedy basis of $X$. Then for each $f\in X$ we have
$$
\|f\| \le C(X,\Psi)\sum_{n=1}^\infty a_n(f)\ff(n)\frac{1}{n}.
$$
\end{Lemma}

\begin{Theorem}\label{QT2} Let $1<p<\infty$.  Then for any quasi-greedy basis $\Psi$ of  $L_p(\Omega,\mu)$ and any $\al \in [0,1]$ we have  
\be\label{Q1}
\ga_m(\al,p,\Psi)  \le C(p,\Psi) m^{-\al(1-1/p^*)} .
\ee
\end{Theorem}
\begin{proof} The proof is based on Theorem \ref{QT1}. By the Definition \ref{QD1} of the quasi-greedy basis we have for $f_m :=f- G_m(f,\Psi)$
\be\label{Q2}
\|f_m\|_p \le C(\Psi,L_p)\|f\|_p.
\ee
By Theorem \ref{QT1} and Lemma \ref{NL} we obtain in the case $2\le p <\infty$
\be\label{Q3}
\|f_m\|_p\le C_2 \sum_{n=1}^\infty n^{-1/2}a_{n+m}(f) \le C_2m^{-1/2}\|f\|_{A_1(\Psi)}
\ee
and for $1<p\le 2$
\be\label{Q4}
\|f_m\|_p\le C_4 \sum_{n=1}^\infty n^{1/p-1}a_{n+m}(f) \le C_4m^{1/p-1}\|f\|_{A_1(\Psi)}.
\ee
Combining bounds (\ref{Q2}) -- (\ref{Q4}) we obtain (\ref{Q1}).

\end{proof}

We now discuss a special subclass of quasi-greedy bases, namely, the uniformly bounded quasi-greedy bases. Existence of the uniformly bounded quasi-greedy bases in $L_p([0,1],\la)$, $1<p<\infty$, is known (see \cite{VTsparse}, p.83, Theorem 3.3.5 and p.76, Theorem 3.2.20). 

\begin{Theorem}\label{QT3} Let $1<p<\infty$.  Then for any uniformly bounded quasi-greedy basis $\Psi$ of  $L_p(\Omega,\mu)$ and any $\al \in [0,1]$ we have  
\be\label{Q1a}
\ga_m(\al,p,\Psi)  \asymp  m^{-\al/2} .
\ee
\end{Theorem}
\begin{proof}
 The upper bounds can be proved in the same way as Theorem \ref{QT2} has been proved. 
Note that in the case $2\le p<\infty$ the upper bounds in Theorem \ref{QT3} follow directly from Theorem \ref{QT2}. In the case $1<p\le 2$ we use instead of Theorem \ref{QT1} the following Proposition \ref{QP1}
(see \cite{VTsparse}, p.77, Proposition 3.2.22).

\begin{Proposition}\label{QP1} Let $\Psi$ be a uniformly bounded quasi-greedy basis $\Psi=\{\psi_j\}_{j=1}^\infty$ in $L_p$, $1<p<\infty$. Then we have for $f\in L_p$
\begin{equation}\label{Q5}
 C'_1(p,\Psi)\sup_n n^{1/2}a_n(f) \le  \|f\|_p\le C'_2(p,\Psi) \sum_{n=1}^\infty n^{-1/2}a_n(f).
\end{equation}
 \end{Proposition}
 
 We now prove the lower bounds. This proof is based on the following Proposition \ref{QP2} (see \cite{VTsparse}, p.76, Proposition 3.2.21).

\begin{Proposition}\label{QP2} Let $\Psi$ be a uniformly bounded quasi-greedy basis $\Psi=\{\psi_j\}_{j=1}^\infty$ in $L_p$, $1<p<\infty$. Then $\Psi$ is democratic with fundamental function $\varphi(m,\Psi,L_p)\asymp m^{1/2}$.
\end{Proposition}

Proposition \ref{QP2} means (see Definition \ref{QD2} above) that there exist two positive constants $c'(p,\Psi)$ and $C'(p,\Psi)$ such  that for any subset of indexes $\La$, $|\La| =m$, we have
\be\label{Q6}
c'(p,\Psi) m^{1/2} \le \|\sum_{k\in \La} \psi_k\|_p \le C'(p,\Psi) m^{1/2}.
\ee
As an example we take 
$$
f:= \sum_{k=1}^{2m} \psi_k\quad \text{and}\quad G_m(f,\Psi) := \sum_{k=1}^{m}\psi_k.
$$
Then we have
$$
\|f\|_p \le C'(p,\Psi)(2m)^{1/2},\quad \|f-G_m(f,\Psi)\|_p \ge c'(p,\Psi)m^{1/2},\quad \|f\|_{A_1(\Psi)} \le2m.
$$
Combining the above relations we obtain the required lower bounds.
\end{proof}

\begin{Remark}\label{QR1} In the case $2\le p<\infty$ the lower bounds of Theorem \ref{QT3} show that
Theorem \ref{QT2} is sharp in this case.
\end{Remark}

\section{A criterion for good rate of approximation}
\label{cr}

We prove here necessary and sufficient conditions on a basis $\Psi$, which provide good rate of approximation for elements from $A_1(\Psi)$. We define the following property of a basis $\Psi$.

{\bf Unconditional bound with parameters $a,K$ (UB($a,K$)).} We say that a basis $\Psi$ of $X$  
satisfies unconditional bound with parameters $a,K$ if for any finite set $\La\subset \N$ and any choice of signs 
$\ep_k=\pm 1$, $k\in\La$, we have 
\be\label{cr1}
\|\sum_{k\in\La} \ep_k\psi_k\|_X \le K|\La|^a.
\ee

It is clear that our condition $\|\psi_k\|_X\le 1$, $k=1,2,\dots$, guarantees that $\Psi$ satisfies UB($1,1$). We restrict $a\in [0,1)$. Obviously, an orthonormal basis of a Hilbert space satisfies UB($1/2,1$). There are 
non-trivial examples of UB($a,K$) bases. For instance, it is known (see \cite{T14}, \cite{VTsparse}, Ch.2, and \cite{VTbookMA}, Ch.8) that for $1<p<\infty$ the univariate Haar basis normalized in $L_p$ satisfies UB($1/p,C(p)$). The reader can find a discussion of results on sparse approximation, including some historical comments, in \cite{VTbookMA}, Ch.9. 

\begin{Theorem}\label{crT1} {\bf (A)} Let $b\in (0,1]$. Suppose that $\Psi$ is such that for any $f\in A_1(\Psi)$ there exists a permutation $\varrho \in D(f)$ with the property
\be\label{cr2}
\|f-G_m(f,\Psi,\varrho)\|_X \le m^{-b},\quad m=1,2,\dots.
\ee
Then $\Psi$ satisfies UB($1-b,2$). 

{\bf (B)} Assume that $\Psi$ satisfies UB($a,K$), $a\in [0,1)$. Then for any $f\in A_1(\Psi)$ and any permutation $\varrho \in D(f)$ we have
\be\label{cr3}
\|f-G_m(f,\Psi,\varrho)\|_X \le C(a)Km^{a-1},\quad m=1,2,\dots.
\ee
\end{Theorem}
\begin{proof} First, we prove part {\bf (A)}. Introduce some notations. For $m\in\N$ let $V(m)$ denote the 
set of vertexes of the unit cube $[-1,1]^m$. In other words $V(m)$ consists of all vectors $\bv = (v_1,\dots,v_m)$ such that $v_j$ takes values $1$ or $-1$ for $j=1,\dots,m$. For $\La =\{k_j\}_{j=1}^m \subset \N$ and  $\bv \in V(m)$ denote
$$
D_{\La,\bv} := \sum_{j=1}^m v_j \psi_{k_j}.
$$
For given $\La$ and $\bv$ consider the function
$$
f = D_{\La,\bv} + (1+\delta) D_{Q,\bu}
$$
where $\delta>0$, $Q\subset \N$ is such that $|Q|=m$ and $\La\cap Q =\emptyset$, $\bu$ is any vector from  $V(m)$, for instance, $\bu=(1,\dots,1)$. Then $((2+\delta)m)^{-1} f\in A_1(\Psi)$ and
$$
G_m(f,\Psi) = (1+\delta) D_{Q,\bu},\quad \|D_{\La,\bv}\|_X =\|f -G_m(f,\Psi)\|_X \le (2+\delta)m^{1-b}.
$$
Taking into account that $\delta>0$ is arbitrary, we conclude the proof of part {\bf (A)}. 

Second, we prove part {\bf (B)}. The proof is based on the following Lemma \ref{crL1}. 

\begin{Lemma}\label{crL1} Let $\Psi=\{\psi_k\}_{k=1}^\infty$ be a  basis of $X$ satisfying UB($a,K$). Then for each $f\in X$ we have
$$
   \|f\|_X\le 4K \sum_{n=1}^\infty n^{a-1}a_n(f),
$$
where $a_n(f)$ are defined in (\ref{an}).
\end{Lemma}
\begin{proof} Without loss of generality assume that $f$ is normalized in such a way that  $a_1(f)< 1$. Denote ${\cN}_s:=\{n:a_n(f)\ge 2^{-s}\}$, $s\in \N$, and $N_s:=|{\cN}_s|$. 
Note, that for small $s$ the sets $\cN_s$ might be empty. Clearly, for a nonempty $\cN_s$ we have
$\cN_s=[1,N_s]$. 
 We have
\be\label{cr4}
\|f\|_X \le   \sum_{s=1}^\infty\|\sum_{n\in {\cN}_s\setminus{\cN}_{s-1}}c_{k_n}(f)\psi_{k_n} \|_X.
\ee
Our assumption that $\Psi$ satisfies the UB($a,K$) implies that for any finite $\La\subset \N$ and any coefficients $b_k$, $k\in \La$, we have
\be\label{cr5}
\|\sum_{k\in\La} b_k\psi_k\|_X \le (\max_{k\in\La}|b_k|) K|\La|^a.
\ee
Indeed, let $\La = \{k_j\}_{j=1}^n$ and $\bw= (w_1,\dots,w_n)$, $w_j := b_{k_j}(\max_{k\in\La}|b_k|)^{-1}$, $j=1,\dots,n$. Then $\bw \in [-1,1]^n$ and, therefore, $\bw$ is a convex combination of elements from $V(n)$. This proves (\ref{cr5}). Using (\ref{cr5}) we derive from (\ref{cr4})
\begin{equation}\label{cr6}
\|f\|_X\le  2K\sum_{s=1}^\infty2^{-s}  |{\cN}_s\setminus{\cN}_{s-1}|^a 
\end{equation}
 and
\be\label{cr7}
\|f\|_X \le   2K\sum_{s=1}^\infty 2^{-s}N_s^{a} \le   2K\sum_{s=1}^\infty2^{-s}\sum_{n=1}^{N_s}n^{a-1}=
2K\sum_{n=1}^\infty n^{a-1}\sum_{s:N_s\ge n} 2^{-s}.
\ee
Denote 
$$
s(n) := \min\{s\,:\, N_s\ge n\}.
$$
Then 
$$
\sum_{s:N_s\ge n} 2^{-s} \le 2^{-s(n)+1}.
$$
Note that for all $s$ such that $N_s\ge n$ we have
$$
2^{-s} \le a_{N_s} \le a_n.
$$
Therefore, we find from (\ref{cr7}) 
$$
\|f\|_X\le  4K\sum_{n=1}^\infty n^{a-1}a_n(f) .
$$
This completes the proof of Lemma \ref{crL1}.
\end{proof}

We now complete the proof of part {\bf (B)}. We have
$$
\|f-G_m(f,\Psi,\varrho)\|_X = \|\sum_{n=m+1}^\infty c_{k_n}(f)\psi_{k_n}\|_X
$$
and by Lemmas \ref{crL1} and \ref{NL} we continue
$$
\le 4K\sum_{n=1}^\infty n^{a-1}a_{n+m}(f) \le C(a)K m^{a-1}\|f\|_{A_1(\Psi)}.
$$
The proof of Theorem \ref{crT1} is now complete.

\end{proof}

\section{Replacing the $\|f\|_{A_1(\Psi)}$ by a weaker norm}
\label{W}

Consider the following norm for $1\le q<\infty$
\be\label{W1}
\|f\|_{A_q^w(\Psi)} := \left(\sup_{n\in\N} na_n(f)^q\right)^{1/q}. 
\ee
We only use the case $q=1$ in our discussion. Clearly,
\be\label{W2}
\|f\|_{A_1^w(\Psi)} \le \|f\|_{A_1(\Psi)}.
\ee
Define the following analog of the quantity $\ga_m(\al,X,\Psi)$
$$
\ga_m^w(\al,X,\Psi) := \sup_{f\in X: \|f\|_{A_1^w(\Psi)}<\infty} \sup_{\varrho\in D(f)} \frac{\|f-G_m(f,\Psi,\varrho)\|_X}{\|f\|_X^{1-\al} \|f\|_{A_1^w(\Psi)}^\al},\quad \al\in [0,1].
$$
Inequality (\ref{W2}) implies that 
\be\label{W3}
\ga_m(\al,X,\Psi) \le \ga_m^w(\al,X,\Psi). 
\ee

We now explain that the above results on the characteristic $\ga_m(\al,X,\Psi)$ hold for the  characteristic $\ga_m^w(\al,X,\Psi)$ as well. Clearly, because of (\ref{W2}) we only need to discuss the upper bounds for the $\ga_m^w(\al,X,\Psi)$. 
We begin with the trigonometric case. First, we deal with the case $2\le p< \infty$. 
 Here is an analog of Theorem \ref{TrT2}.

\begin{Theorem}\label{WT1} Let $2\le p < \infty$. Then
$$
\ga_m^w(\al,p,\Tr^d) \asymp m^{h(p)-\al/2},\quad \al\in [0,1].
$$
\end{Theorem}
\begin{proof} The lower bound follows from Theorem \ref{TrT2}. We now prove the upper bound. By Theorem \ref{TrT1}
we obtain
\be\label{W4}
\|f - G_m(f,{\mathcal T}^d)\|_p \le (1+3m^{h(p)})\|f\|_p.
\ee
Denote for brevity $H:= \|f\|_{A_1^w(\Tr^d)}$. Then $a_n(f) \le H/n$, $n\in\N$ and by the Paley inequality ($p':= p/(p-1)$)
$$
\|f - G_m(f,{\mathcal T}^d)\|_p \le C(p) \left(\sum_{n=m+1}^\infty a_n^{p'}\right)^{1/p'}
$$
\be\label{W5}
  \le C(p)H \left(\sum_{n=m+1}^\infty  n^{-p'}\right)^{1/p'} \le C_1(p) H m^{-1/p}.
\ee
Raising inequality (\ref{W4}) to the power $1-\al$ and inequality (\ref{W5}) to the power $\al$ and multiplying them, we complete the proof. 
\end{proof}

Here is an analog of Corollary \ref{TrC1}. 

\begin{Corollary}\label{WC1} Let $2\le p< \infty$. Then for any $f\in L_p$ and for each $m\in \N$ we have
$$
\|G_m(f,\Tr^d)\|_p \le C(p)\|f\|_p^{2/p}\|f\|_{A_1^w(\Tr^d)}^{1-2/p}.
$$
\end{Corollary}

Second, we deal with the case $1\le p\le 2$. 
 Here is an analog of Theorem \ref{TrT3}.
 
 \begin{Theorem}\label{WT2} Let $1\le p \le 2$. Then
$$
\ga_m^w(\al,p,\Tr^d) \asymp m^{h(p)-\al/p},\quad \al\in [0,1].
$$
\end{Theorem}
 \begin{proof} As above, we only need to prove the upper bound. The proof goes along the lines of the proof of Theorem \ref{TrT3}. We will not repeat that proof and point out that instead of Lemma \ref{mLqp} we use the following simple Lemma \ref{WL1}. 
 
 \begin{Lemma}\label{WL1} Let $a_1\ge a_2\ge \cdots \ge a_M\ge 0$ and $1\le q< p\le \infty$.
Then for all $m\le M$ one has
$$
\left(\sum_{k=m}^M a_k^p\right)^{1/p} \le C(p/q)m^{-\bt} \left(\sup_{1\le k \le M} ka_k^q\right)^{1/q},\quad \bt:= 1/q-1/p.
$$
\end{Lemma}
 
 \end{proof}
 
 Here is an analog of Corollary \ref{TrC2}. 
 
\begin{Corollary}\label{WC2} Let $1\le p\le 2$. Then for any $f\in L_p$ we have
$$
\|G_m(f,\Tr^d)\|_p \le C\|f\|_p^{p/2}\|f\|_{A_1^w(\Tr^d)}^{1-p/2}.
$$
\end{Corollary}

We now proceed to the other type of systems $\Psi$. 
 
 \begin{Remark}\label{WR1} Theorems \ref{UT1} and \ref{UT1a} hold for the $\ga_m^w(\al,p,\Psi)$ as well. 
 \end{Remark}
 \begin{proof} The proofs of Theorems \ref{UT1} and \ref{UT1a} are based on Lemma \ref{mLqp}. Replacing Lemma \ref{mLqp} by Lemma \ref{WL1} and repeating the arguments from the proofs of Theorems \ref{UT1} and \ref{UT1a},
 we obtain the statement of the Remark \ref{WR1}.
  \end{proof}
 
 In the proofs of Theorems \ref{UT2}, \ref{QT2}, and \ref{QT3} we only use the norm $\|f\|_{A_1(\Psi)}$ for getting the inequality 
 $a_n(f) \le n^{-1}\|f\|_{A_1(\Psi)}$. This means that the proofs of those theorems give the corresponding upper bounds 
 for the characteristic $\ga_m^w(\al,p,\Psi)$.   
  We formulate this observation as a remark.

  \begin{Remark}\label{WR2} Theorems \ref{UT2}, \ref{QT2}, and \ref{QT3} hold for the $\ga_m^w(\al,p,\Psi)$ as well. 
 \end{Remark}

{\bf Acknowledgements.} The author is grateful to Pablo Bern{\'a} Larossa for helpful comments. The author would like to thank the Isaac Newton Institute for Mathematical Sciences, Cambridge, for support and hospitality during the programme DREW01 where work on this paper was completed. This work was supported by EPSRC grant no EP/R014604/1.

\end{document}